\documentclass[12pt]{article}
\usepackage{amsmath, amssymb, amsthm, amsfonts}
\usepackage{indentfirst}
\usepackage[mathscr]{euscript}
\usepackage{amscd}
\usepackage{color}
\usepackage[pdftex]{graphicx}

\numberwithin{equation}{section}
\numberwithin{figure}{section}

\theoremstyle{plain}
\newtheorem{theorem}{Theorem}\numberwithin{theorem}{section}
\newtheorem{lemma}{Lemma}\numberwithin{lemma}{section}
\newtheorem{proposition}{Proposition}\numberwithin{proposition}{section}
\newtheorem{corollary}{Corollary}\numberwithin{corollary}{section}
\theoremstyle{definition}
\numberwithin{definition}{section}

\theoremstyle{plain}
\newtheorem{remark}{Remark}\numberwithin{remark}{section}

\theoremstyle{plain}
\numberwithin{example}{section}

\newcommand{\R}{\mathbb{R}}

\newcommand{\fc}{{\mathfrak{F}}_c}
\newcommand{\fs}{{\mathfrak{F}}_s}
\newcommand{\F}{{}_1F_2}
\newcommand{\bH}{\mathbf{H}}

\title{The zeros of certain Fourier transforms:\\ Improvements of P\'olya's results}

\author{Yong-Kum Cho\footnote{ykcho@cau.ac.kr.
Department of Mathematics, College of Natural Sciences, Chung-Ang University,
84 Heukseok-Ro, Dongjak-Gu, Seoul 06974, Korea.}
\and Young Woong Park\footnote{ywpark1839@gmail.com. Department of Mathematics, College of Natural Sciences, Chung-Ang University,
84 Heukseok-Ro, Dongjak-Gu, Seoul 06974, Korea.}}

\date{}

\begin{document}
\maketitle

\textsc{Abstract.} {\small As for the Fourier transforms of positive and integrable functions
supported in the unit interval, we make a list of improvements for P\'olya's results
on the distribution of their positive zeros and give new sufficient conditions under which those zeros
are simple and regularly distributed.
As an application, we take the two-parameter family of
beta probability density functions defined by
\begin{equation*}
f(t)= \frac{1}{B(\alpha, \beta)}\, (1-t)^{\alpha-1} t^{\beta-1},\quad 0<t<1,
\end{equation*}
where $\,\alpha>0,\,\beta>0,$ and specify the distribution of zeros of the
associated Fourier transforms for some region of $(\alpha, \beta)$ in the first quadrant
which turns out to be much larger than the region where P\'olya's results are applicable.
}

\medskip

\textsc{Keywords.} {\small Distribution of zeros, Entire functions, Fourier transforms,}

{\small Partial fractions, Positivity, Wronskian inequalities.}

\medskip

\textsc{2020 Mathematics Subject Classification.} {\small  30C15, 33C20, 42A38.}

\section{Introduction}

In his fundamental work \cite{P} published in 1918, G. P\'olya, in collaboration with A. Hurwitz,
investigated the zeros of Fourier transforms of type
\begin{equation}\label{P1}
\left\{\begin{aligned}
&{U(x) =\int_0^1 f(t)\cos xt\, dt,}\\
&{V(x)  =\int_0^1 f(t) \sin xt\, dt}\end{aligned}\right.\qquad (x>0),
\end{equation}
where $f(t)$ is assumed to be positive and integrable.
In a simplified version, P\'olya's main results may be summarized as follows.

\begin{itemize}
\item[{}] {\bf{Theorem A.}}
Let $f(t)$ be a positive increasing function for $\,0<t<1\,$
such that it is in the general case\footnote{
$f(t)$ is said to be in the exceptional case
if there exist a finite number of rational points, say, $\,0 = r_0<r_1<\cdots<r_N=1,$ such that
$f(t)$ is constant in each of the open intervals $\,\big(r_{k-1},\,r_k\big),\, k=1,\cdots,N.$ Otherwise, $f(t)$ is said to be
in the general case.} and the integral $\,\int_0^1 f(t) dt\,$ exists.

\begin{itemize}
\item[(A1)] $\,U(x)>0\,$ for $\,0<x\le \pi/2,$ $U(x)$
has one and only one simple zero in each of the intervals
$\,\big((k-1/2)\pi,\,(k+1/2)\pi\big),\, k=1,2,\cdots,$
and no positive zeros elsewhere. If, in addition, $f(t)$ is convex
and the right derivative $f'(t)$ is in the general case, then $U(x)$ has a unique simple zero in each of the intervals
$\,\big((k-1/2)\pi,\,k\pi\big).$

\item[(A2)] $\,V(x)>0\,$ for $\,0<x\le\pi,$ $V(x)$
has one and only one simple zero in each of the intervals
$\,\big(k\pi,\,(k+1)\pi\big),\,k=1,2,\cdots,$
and no positive zeros elsewhere. If, in addition, $f(t)$ is convex and has the limiting value
$\,f(0+) = \lim_{t\to 0+} f(t) = 0,$ then $V(x)$ has a unique simple zero in each of the intervals
$\,\big(k\pi, \,(k+1/2)\pi\big).$
\end{itemize}
\end{itemize}

\begin{itemize}
\item[{}]{\bf Theorem B.}
Let $f(t)$ be a positive, decreasing and concave function for $\,0<t<1\,$
such that $-f'(t)$ is in the general case and the integral $\,\int_0^1 f(t) dt\,$ exists.
Then $\,U(x)>0\,$ for $\,0<x\le\pi,$ $U(x)$ has one and only one simple zero
in each of the intervals
$\,\big(k\pi,\,(k+1)\pi\big),\,k=1,2,\cdots,$
and no positive zeros elsewhere.
\end{itemize}

In a more extensive setting, P\'olya also investigated the case when Fourier transforms never vanish
and obtained the following results (see also \cite{T}).

\begin{itemize}
\item[{}]{\bf Theorem C.}
If $f(t)$ is a positive, continuous and decreasing function for
$\,0<t<\infty\,$ such that the integral $\,\int_0^\infty f(t) dt\,$ exists, then
\begin{equation*}
\left\{\begin{aligned}
&{\int_0^\infty f(t)\sin xt\,dt>0,}\\
&{\int_0^\infty f(t)\cos xt\,dt>0,}\end{aligned}\right.\qquad (x>0),
\end{equation*}
where the second inequality holds true under the additional assumption that $f(t)$ is convex
for $\,0<t<\infty.$
\end{itemize}

Due to explicit and precise descriptions on the nature of zeros under such simple conditions,
P\'olya's results have far-reaching applications in various fields of mathematical sciences
which deal with oscillatory phenomena by means of Fourier transforms of type \eqref{B1} or its variants.
Nevertheless, as will be illustrated with numerous examples in the subsequent developments,
there still leave much room for improvement. A notable drawback is that none of P\'olya's results
is applicable when the density function is not monotone.

In this paper we aim at making a list of improvements which yield sharper bounds of zeros
under certain circumstances. In particular, we shall prove that
the monotonicity or convexity condition on $f(t)$ could be replaced by the positivity of one of Fourier transforms
of $f(1-t)$ or both in such a way that P\'olya's results or its analogues continue to hold true.

As the main application of our results, we shall consider the two-parameter family of
beta probability density functions defined by
\begin{equation}\label{B0}
f(t)= \frac{1}{B(\alpha, \beta)}\, (1-t)^{\alpha-1} t^{\beta-1},\quad 0<t<1,
\end{equation}
where $B(\alpha, \beta)$ denotes the beta function for $\,\alpha>0, \,\beta>0.$
By making use of the known positivity results, we shall specify the distribution of zeros of the associated Fourier transforms
for certain region of $(\alpha, \beta)$ which turns out to be much larger
than the region of monotonicity.

For the sake of convenience, we shall often use the notation
\begin{equation}\label{P1}
\left\{\begin{aligned}
&{\fc[f(t)](x) =\int_0^1 f(t)\cos xt\, dt,}\\
&{\fs[f(t)](x) =\int_0^1 f(t) \sin xt\, dt.}\end{aligned}\right.
\end{equation}
While there are a great deal of literatures related with P\'olya's work,
we refer to the recent survey paper \cite{DR} of Dimitrov and Rusev which not only contains
relevant references but also gives outlines of the paper \cite{P} and indicates connections with other developments.

\section{Positivity}
A consequence of Theorem C is that
\emph{if $f(t)$ is a positive, continuous and decreasing function for $\,0<t<1\,$ such that
the integral $\int_0^1 f(t) dt$ exists, then its Fourier sine transform $V(x)$ is positive for all $\,x>0.$}
By a modification of Lommel's arguments for the existence of zeros of Bessel functions (\cite[15.2]{Wa}),
it is possible to remove the integrability condition as well as continuity
which may be too restrictive to include the general case of interest.

\begin{theorem}\label{theorem1}
If $f(t)$ is a positive decreasing function for $\,0<t<1\,$ such that it is in the general case, then
$\,V(x)>0\,$ for all $\,x>0.$
\end{theorem}

\begin{proof}
On writing $\, V(x) = \pi/x \cdot S(x),$ where
$$S(x) = \int_0^{x/\pi} f\left(\frac{\pi t}{x}\right)\sin \pi t\,dt,$$
it suffices to prove the strict positivity of $S(x)$ for $\,x>0.$ Clearly,
$\,S(x)>0\,$ for $\,0<x\le\pi.$ For a positive integer $m$ and $\,0\le\theta\le 1,$ put
$$ x= (m+\theta)\pi, \quad \phi(t) = f\left(\frac{t}{m+\theta}\right).$$
By decomposing the integral and changing variables, we have
\begin{align}\label{S1}
S(x) &= \left(\sum_{k=1}^m \int_{k-1}^k + \int_m^{m+\theta}\right) \phi(t) \sin \pi t\,dt\nonumber\\
&= \sum_{k=1}^m (-1)^{k+1} \int_0^1 \phi(k-1+t) \sin \pi t\,dt\nonumber\\
&\qquad + (-1)^m \int_0^\theta \phi(m+t) \sin \pi t\,dt.
\end{align}

\begin{itemize}
\item[(a)] If $\, m=1,$ we write \eqref{S1} into the form
\begin{align*}
S(x) = \int_0^\theta \big[\phi(t) - \phi(1+t)\big] \sin \pi t\,dt + \int_\theta^1 \phi(t) \sin \pi t\,dt.
\end{align*}
Since $\phi(t)$ is positive, deceasing on $(0, 1)$ and does not belong to the exceptional case,
it implies that $\,S(x)>0.$

\item[(b)] If $m$ is odd with $\,m\ge 3,$ we put \eqref{S1} into the form
\begin{align*}
S(x) &= \int_0^1 K(\phi;t)\sin \pi t\,dt  +\int_\theta^1\phi(m-1+t)\sin \pi t\,dt\\
&\qquad +\int_0^\theta  \big[\phi(m-1+t)-\phi(m+t)\big]\sin \pi t\,dt,
\end{align*}
where $K(\phi;t)$ denotes the kernel
$$ K(\phi;t) = \big[\phi(t) - \phi(1+t)\big] + \cdots + \big[\phi(m-3+t)-\phi(m-2+t)\big],$$
whence $\,S(x)>0\,$ due to the same reasons as above.

\item[(c)] If $m$ is even, we write \eqref{S1} as
\begin{align*}
S(x) &= \int_0^1 L(\phi;t)\sin \pi t\,dt  +\int_0^\theta  \phi(m+t)\sin \pi t\,dt,
\end{align*}
where $L(\phi;t)$ denotes the kernel
$$ L(\phi;t) = \big[\phi(t) - \phi(1+t)\big] + \cdots + \big[\phi(m-2+t)-\phi(m-1+t)\big],$$
which clearly shows that $\,S(x)>0.$
\end{itemize}

In summary, it is verified that $\,S(x)>0\,$ for $\,m\pi\le x\le (m+1)\pi,$ which combined with
positivity for $\,0<x\le\pi\,$ proves the theorem.
\end{proof}

\smallskip

\begin{remark}
An inspection on the proof reveals that $\,V(x)\ge 0\,$ for all $\,x>0\,$
even when $f(t)$ belongs to the exceptional case.
\end{remark}

\smallskip

An application of Theorem \ref{theorem1} yields
\begin{align}\label{G1}
\int_0^x t^{-\alpha}\sin t\,dt &= x^{1-\alpha}\int_0^1 t^{-\alpha}\sin xt\,dt\nonumber\\
&=\frac{x^{2-\alpha}}{2-\alpha}\cdot \F\left[\begin{array}{c}
(2-\alpha)/2\\ 3/2, \,(4-\alpha)/2\end{array}\biggr| - \frac{x^2}{4}\right]>0
\end{align}
for all $\,x>0\,$ when $\,0<\alpha<2.$ We note that
$\,f(t) = t^{-\alpha}\,$ is positive and decreasing for $\,0<t<1\,$
but integrable only for $\,0<\alpha<1.$
We refer to \cite[(5.16)]{CY} for a different proof for the positivity of the above generalized sine integral
on the basis of Gasper's sums of squares method \cite{G}.

\begin{corollary}\label{corollary1}
If $f(t)$ takes the form
$$ f(t) =\int_t^1 g(s) ds\qquad(0<t<1),$$
where $g(s)$ is a positive, continuous and decreasing function for $\,0<s<1\,$ such that it is in the general case,
then $\,U(x)>0\,$ for all $\,x>0.$
In particular, if $f(t)$ is a positive, decreasing and convex function for $\,0<t<1\,$ such that
it has the limiting value $\,f(1-) = \lim_{t\to 1-} f(t)=0\,$ and $f'(t)$ is in the general case, then
$\,U(x)>0\,$ for all $\,x>0.$
\end{corollary}

The proof of the first part is immediate due to
$$\fc\left[\int_t^1 g(s) ds\right](x) = \frac 1x \int_0^1 g(t) \sin xt\,dt,$$
a simple consequence of integrating by parts. The second part follows from the first part by
taking $\,g(t) = f'(t),$ the right derivative of $f(t)$.

It should be emphasized that the boundary limit condition $\,f(1-) =0\,$ is not necessary
for the positivity of $U(x)$. For example, as a special case of Askey-Szeg\"o problem,
it is shown in \cite[Theorem 4.5]{CCY} that
\begin{align}\label{G2}
\int_0^x t^{-\alpha}\cos t\,dt &= x^{1-\alpha}\int_0^1 t^{-\alpha}\cos xt\,dt\nonumber\\
&=\frac{x^{1-\alpha}}{1-\alpha}\cdot \F\left[\begin{array}{c}
(1-\alpha)/2\\ 1/2, \,(3-\alpha)/2\end{array}\biggr| - \frac{x^2}{4}\right]>0
\end{align}
for all $\,x>0\,$ when $\,1/3\le\alpha<1\,$ for which the function $\,f(t) = t^{-\alpha}\,$
is positive, decreasing and convex for $\,0<t<1\,$ but $\,f(1-) =1.$

If $f(t)$ has bounded second derivative, however, then the condition turns out to be
necessary. In fact, an integration by parts gives
\begin{align*}
U(x) &= f(1-)\,\frac{\sin x}{x} +\frac{1}{x}\int_0^1 f'(t) \sin xt\,dt\\
&=  f(1-)\,\frac{\sin x}{x} + O\left(\frac{1}{x^2}\right)\quad\text{as}\quad x\to \infty,
\end{align*}
due to the Riemann-Lebesgue lemma, so that $U(x)$ must have an infinity of
zeros for sufficiently large $x$ as long as $\,f(1-)>0.$

\section{Method of partial fractions}
This section aims to revisit the method of partial fractions used decisively
in P\'olya's investigations for the sake of subsequent applications.

For any function $f(t)$ defined for $\,0<t<1,$ if the integral $\int_0^1 |f(t)|dt$ converges,
then it is evident that Fourier transforms
\begin{equation}
\left\{\begin{aligned}
&{U(z) =\int_0^1 f(t)\cos zt\, dt,}\\
&{V(z)  =\int_0^1 f(t) \sin zt\, dt}\end{aligned}\right.\qquad (z\in\mathbb{C})
\end{equation}
are entire functions. In consideration of the Fourier cosine series
$$f(t) \,\sim\,\frac{a_0}{2} + \sum_{k=1}^\infty a_k \cos k\pi t,\quad a_k = 2 U(k\pi),$$
if it is assumed, for instance, that the series converges to $f(t)$ uniformly on every bounded subinterval
of $(0, 1)$, then termwise integrations lead to
\begin{equation}\label{PE1}
\frac{\,U(z)\,}{\,\sin z\,} = \frac{U(0)}{z} +
\sum_{k=1}^\infty (-1)^k U(k\pi)\left(\frac{1}{z-k\pi} + \frac{1}{z+ k\pi}\right),
\end{equation}
the partial fraction formula discovered by Hurwitz \cite{P}. By using this formula,
Hurwitz proved that if the Fourier coefficients $(a_k)$ alternate in sign with $\,a_0>0,$
then $U(z)$ has only real simple zeros and each of the intervals
$\,\big(k\pi,\,(k+1)\pi\big)\,$ contains exactly one zero, where $\, k=1,2,\cdots.$

From a complex analysis point of view, P\'olya proved that
\eqref{PE1} remains valid under the sole condition $\,\int_0^1 |f(t)|dt<\infty.$
In fact, he noticed that the residue at the simple pole $k\pi$ of $\,U(z)/\sin z\,$ is given by
$$\lim_{z\to k\pi}\frac{(z-k\pi) U(z)}{\sin z} = (-1)^k U(k\pi),\quad k\in\mathbb{Z},$$
and hence the difference between $\,U(z)/\sin z\,$ and the series on the right side of \eqref{PE1}
is an entire function. By expressing $\,U(z)/\sin z\,$ as a contour integral and
exploiting periodicity of $\sin z$, he verified that the difference is bounded in $\mathbb{C}$,
whence it reduces to a constant according to Liouville's theorem. Since
the difference is an odd function, the constant must be zero,
which completes P\'olya's proof for Hurwitz's formula \eqref{PE1}.

In a similar manner, P\'olya obtained two additional formulas
\begin{align}
\frac{U(z)}{z\cos z} &= \frac{U(0)}{z} + \sum_{k=1}^\infty\frac{(-1)^k U\left[(k-1/2)\pi\right]}{(k-1/2)\pi}\nonumber\\
&\qquad\qquad\quad\times\left[\frac{1}{z-(k-1/2)\pi} + \frac{1}{z+ (k-1/2)\pi}\right],\label{PE2}\\
\frac{V(z)}{z\sin z} &= \frac{V'(0)}{z} + \sum_{k=1}^\infty \frac{(-1)^k V(k\pi)}{k\pi}
\left(\frac{1}{z-k\pi} + \frac{1}{z+ k\pi}\right),\label{PE3}
\end{align}
both of which are valid under the condition $\,\int_0^1 |f(t)|\,dt<\infty.\,$

When the coefficients for $\,k\ge 1\,$ keep the same sign,
these partial fraction expansions can be used in various ways to determine the existence and nature
of zeros. To explain the idea of Hurwitz and P\'olya briefly, note that all of the partial sums
are rational functions of the form
$$S_N(z) = \frac{A_0}{z} + \sum_{k=1}^N A_k\left(\frac{1}{z-a_k} + \frac{1}{z+ a_k}\right),$$
where $\,0=a_0<a_1<a_2<\cdots <a_N\,$ and $A_k$ coincides with the residue at the pole $a_k$.
Suppose, for instance, that $\,A_0>0\,$ and $\,A_k>0\,$ for all $\,k\ge 1$.
Then it is elementary to see that $S_N(x)$ as a function of $\,x>0\,$ decreases
strictly from $+\infty$ to $-\infty$ on $\,\big(a_{k-1}, \,a_k\big),\,k=1,\cdots, N,$
so that it has a unique simple zero in each of these intervals.
As the possibility that the zeros of $S_N(z)$ move towards
end-points $a_k$'s are obviously excluded, one may conclude by letting $\,N\to\infty\,$
that the corresponding Fourier transform has only real simple zeros
whose positive zeros are distributed in the same way as described.

Our analysis will be based on Wronskian inequalities. To be more specific,
let us denote the Wronskian determinant of $f(z), g(z)$ by
$$ W[f(z),\, g(z)] = f(z)g'(z)-f'(z) g(z).$$
As to Hurwitz's formula \eqref{PE1}, if we multiply by $z$ and differentiate,
where differentiating termwise is justified due to the uniform convergence
of the derived series on every compact subset of $\,\mathbb{C}\setminus\left\{\pm k\pi : k\in\mathbb{N}\right\},$ we obtain
\begin{equation}\label{PW1}
W\left[U(z), \,\frac{\sin z}{z}\right] = \frac{\,4\sin^2 z\,}{z}
\sum_{k=1}^\infty \frac{\,(-1)^k U(k\pi)\, k^2\pi^2\,}{\left(z^2- k^2\pi^2\right)^2}.
\end{equation}
Since the singularities at $\,z =0\,$ or $\, z = k\pi,\,k=\pm 1, \pm 2, \cdots\,$ are easily seen to be removable,
this formula remains valid for all $\,z\in\mathbb{C}.\,$

Proceeding in the same way, we obtain from \eqref{PE2}, \eqref{PE3}
\begin{equation}\label{PW2}
W\big[U(z),\, \cos z\big] = 4z \cos^2 z\sum_{k=1}^\infty \frac{\,(-1)^k U\big[(k-1/2)\pi)\big]\, (k-1/2)\pi\,}
{\big[z^2- (k-1/2)^2 \pi^2\big]^2},\end{equation}
\begin{equation}\label{PW3}
W\big[V(z),\, \sin z\big] =4z \sin^2 z \sum_{k=1}^\infty \frac{\,(-1)^k V(k\pi)\,k\pi\,}
{\left(z^2- k^2\pi^2\right)^2},
\end{equation}
both of which are valid for all $\,z\in\mathbb{C}.\,$

For a smooth oscillatory function $f(x)$ defined on an open interval $I$, it is known
in Sturm's theory of oscillations, or readily verified otherwise,
that if the Wronskian $W[f(x), g(x)]$ keeps constant sign on $I$, then $g$
has an infinity of zeros which are simple and interlaced with the zeros of $f(x)$.
Due to our Wronskian formulas \eqref{PW1}, \eqref{PW2}, \eqref{PW3}, it is thus immediate to obtain
the following set of criteria on the nature and distribution of zeros
which strengthen to some extent the original ones implicitly stated in \cite{P}.

\begin{theorem}\label{theoremP}{\rm{(P\'oly's criteria)}}
Let $f(t)$ be a positive function for $\,0<t<1\,$ such that it is in the general case and the integral
$\int_0^1 f(t) dt$ exists.

\begin{itemize}
\item[\rm(P1)]
If $\,(-1)^k U\big[(k-1/2)\pi\big]<0\,$ for all $k$, then
$$W\big[\cos x,\, U(x)\big]>0\quad\text{for all} \quad x>0$$
and hence $U(x)$ has an infinity of positive zeros which are all simple and interlaced with
the zeros of $\cos x$ in such a way that $U(x)$ has a unique simple zero in each of the intervals
$$\,\big((k-1/2)\pi,\,(k+1/2)\pi\big),\quad k=1,2,\cdots,\,$$
and no positive zeros elsewhere.

\item[\rm(P2)]
If $\,(-1)^k U(k\pi)>0\,$ for all $k$, then
$$W\left[U(x),\,\frac{\sin x}{x}\right]>0\quad\text{for all} \quad x>0$$
and hence $U(x)$ has an infinity of positive zeros which are all simple and interlaced with
the zeros of $\sin x$ in such a way that $U(x)$ has a unique simple zero in each of the intervals
$$\,\big(\pi/2,\,\pi\big),\,\,\,\big(k\pi,\,(k+1)\pi\big),\quad k=1,2,\cdots,\,$$
and no positive zeros elsewhere.

\item[\rm(P3)]
If $\,(-1)^k V(k\pi)<0\,$ for all $k$, then
$$W\big[\sin x,\, V(x)\big]>0\quad\text{for all} \quad x>0$$
and hence $V(x)$ has an infinity of positive zeros which are all simple and interlaced with
the zeros of $\sin x$ in such a way that $V(x)$ has a unique simple zero in each of the intervals
$$\,\big(k\pi,\,(k+1)\pi\big),\quad k=1,2,\cdots,\,$$
and no positive zeros elsewhere.
\end{itemize}
\end{theorem}

\begin{remark}\label{remarkP}
As an alternative of (P2), if $\,(-1)^k U(k\pi)<0\,$ for all $k$, then
$U(x)$ has one and only one simple zero
in each of the intervals
$$\,\big(k\pi,\,(k+1)\pi\big),\quad k=1,2,\cdots.\,$$
\end{remark}

\section{Use of positivity}
To apply P\'olya's criteria effectively, it is essential to know
whether the values of $\,U(x), V(x)\,$ at the positive zeros of $\,\cos x\,$ or $\,\sin x\,$ alternate in sign,
which may cause great difficulties in practice. If available, it is advantageous to exploit
the positivity of Fourier transform of $f(1-t)$.

\begin{lemma}\label{lemma1}
Let $f(t)$ be a positive function for $\,0<t<1\,$ such that it is in the general case and the integral
$\int_0^1 f(t) dt$ exists. Define
\begin{equation}\label{FS}
\left\{\begin{aligned}
&{U_s(x) =\int_0^1 f(1-t)\cos xt\, dt,}\\
&{V_s(x)  =\int_0^1 f(1-t) \sin xt\, dt}\end{aligned}\right.\qquad (x>0).
\end{equation}

\begin{itemize}
\item[\rm(i)]
If $\,U_s(x)>0\,$ for all $\,x>0,$ then the inequalities
$$(-1)^k U(k\pi)>0, \quad (-1)^{k} V\big[(k-1/2)\pi\big]<0$$
hold true for each positive integer $k$. Moreover, $U(x), V(x)$
can not have common positive zeros.

\item[\rm(ii)]
If $\,V_s(x)>0\,$ for all $\,x>0,$ then the inequalities
$$(-1)^{k} U\big[(k-1/2)\pi\big]<0, \quad (-1)^{k} V(k\pi)<0$$
hold true for each positive integer $k$. Moreover, $U(x), V(x)$
can not have common positive zeros.
\end{itemize}
\end{lemma}

The proofs are immediate in view of the relations
\begin{align*}
U_s(x) &= \cos x\,U(x) +\sin x\,V(x),\\
V_s(x) &= \sin x\,U(x) -\cos x\,V(x).
\end{align*}

Owing to Lemma \ref{lemma1}, we may rephrase P\'olya's criteria as follows.

\begin{corollary}\label{corollaryP}
Let $f(t)$ be a positive function for $\,0<t<1\,$ such that it is in the general case and the integral
$\int_0^1 f(t) dt$ exists.

\begin{itemize}
\item[\rm(i)] Suppose that $\,V_s(x)>0\,$ for all $\,x>0.$
Then $U(x)$ has exactly one simple zero in each of the intervals
$\,\big((k-1/2)\pi,\,(k+1/2)\pi\big)\,$ and so does $V(x)$ in each of the intervals
$\,\big(k\pi,\,(k+1)\pi\big),$ where
$\,k=1,2,\cdots.$ Moreover, both $U(x), V(x)$ have no positive zeros elsewhere and share no common zeros.

\item[\rm(ii)] Suppose that $\,U_s(x)>0\,$ for all $\,x>0.$ Then
$U(x)$ has exactly one simple zero in each of the intervals
$\,\big(\pi/2,\,\pi\big),\,\,\big(k\pi,\,(k+1)\pi\big)\,$ and
$V(x)$ has at least one zero in each of the intervals
$$\,\big(\pi,\,3\pi/2\big),\,\, \big((k+1/2)\pi,\,(k+3/2)\pi\big),$$ where $\,k=1,2,\cdots.$
Moreover, both $U(x), V(x)$ have no positive zeros elsewhere and share no common zeros.

\item[\rm(iii)] Suppose that $\,V_s(x)>0, \,U_s(x)>0\,$ for all $\,x>0.$
Then $U(x)$ has exactly one simple zero in each of the intervals
$\,\big((k-1/2)\pi,\,k\pi\big)\,$ and so does $V(x)$ in each of the intervals
$\,\big(k\pi,\,(k+1/2)\pi\big),$ where
$\,k=1,2,\cdots.$ Moreover, $U(x), V(x)$ have no positive zeros elsewhere and share no common zeros.
\end{itemize}
\end{corollary}

\begin{remark}
If $f(t)$ is in the exceptional case, then it is easy to prove that $U(x), V(x)$ have infinitely many
common positive zeros.
\end{remark}

While we shall present a class of examples in the last section,
we remark that P\'olya's original theorems can be proved easily by applying the above criteria.
To describe briefly, let us assume that $f(t)$ is a positive increasing function
for $\,0<t<1\,$ such that it is in the general case and $\int_0^1 f(t) dt$ exists. Since $f(1-t)$
changes monotonicity, Theorem \ref{theorem1} shows that
$\,V_s(x)>0\,$ for all $\,x>0\,$ and the first parts of (A1), (A2) of Theorem A follow from part (i).
If $f(t)$ is convex at the same time with $\,f(0+) =0,$ then $f(1-t)$ is also convex with $\,f(1-)=0.$
It follows from Corollary \ref{corollary1} that $\,U_s(x)>0\,$ for all $\,x>0\,$
and the second part of (A2) follows from part (iii).

On integrating by parts, it is trivial to see that
$$U(k\pi) = \frac{1}{k\pi}\int_0^1 \left(-f'(t)\right)\sin k\pi t\,dt,\quad k=1, 2, \cdots,$$
valid due to the integrability condition on $f(t)$, as P\'olya observed, and the second part of (A1) and Theorem B follow by
obvious modifications.

Concerning the location of the first positive zero of $U(x)$, however, we may
improve the upper bound based on the following observation.

\begin{lemma}\label{lemma2}
Let $f(t)$ be positive for $\,0<t<1\,$
and in the general case.

\begin{itemize}
\item[\rm(i)] If $f(t)$ is decreasing, then $\,U(x)>0\,$ for $\,0<x\le\pi.$
\item[\rm(ii)] If $f(t)$ is increasing, then $U(x)$ has at least one zero for $\,\pi/2<x<\pi.$
\end{itemize}
\end{lemma}

\begin{proof}
Setting $\,x = \pi(1+\theta)/2,$ where
$\,0\le \theta\le 1,$ we have
\begin{align*}\label{CV11}
(1+\theta)U(x) &=\left(\int_0^1 +\int_1^{1+\theta}\right)
f\left(\frac{t}{1+\theta}\right)\cos\left(\frac{\pi t}{2}\right)dt\nonumber\\
&=\int_0^\theta \left[f\left(\frac{1-t}{1+\theta}\right) - f\left(\frac{1+t}{1+\theta}\right)\right]
\sin\left(\frac{\pi t}{2}\right)dt\nonumber\\
&\qquad\qquad + \int_\theta^1 f\left(\frac{1-t}{1+\theta}\right)\sin\left(\frac{\pi t}{2}\right)dt.
\end{align*}

If $f(t)$ is decreasing, then it shows that $\,U(x)>0\,$
for $\,\pi/2\le x\le \pi.$ Since $\,U(x)>0\,$ for $\,0<x\le\pi/2,$
we conclude that $\,U(x)>0\,$ for $\,0<x\le\pi.$ If $f(t)$ is increasing,
then we take $\,\theta =1\,$ to find that $\,U(\pi)<0.$ Since $\,U(\pi/2)>0,$
the result of part (ii) follows by the intermediate value property.
\end{proof}

The first part of (A1), Theorem A, must be revised as follows.

\begin{corollary}\label{corollary0}
If $f(t)$ is a positive increasing function for $\,0<t<1\,$
such that it is in the general case and the integral $\,\int_0^1 f(t) dt\,$ exists, then
$U(x)$ has exactly one simple zero in each of the intervals
$$\big(\pi/2,\,\pi),\,\,\big((k+1/2)\pi,\,(k+3/2)\pi\big),\quad k=1,2,\cdots,$$
and no other positive zeros.
\end{corollary}

\section{Zeros of derivatives}
By Rolle's theorem, if the Fourier transform has an infinity of zeros, then it is evident that
its derivative has at least one zero between two consecutive zeros. On inspecting the method of Wronskians,
however, it is possible to obtain more precise information on the zeros of derivatives.

Given a positive function $f(t)$ defined for $\,0<t<1,$ if it is integrable, then
differentiating under the integral sign gives
\begin{equation}
U'(x) = -\int_0^1 t f(t)\sin xt\,dt,
\quad V'(x) =  \int_0^1 t f(t) \cos xt\,dt.
\end{equation}
If $f(t)$ is increasing, then $tf(t)$ is also increasing and in the general case.
It follows from Theorem A that $U'(x)$ has exactly one simple zero in each of the
intervals $\,\big(k\pi,\,(k+1)\pi\big).$ Likewise, $V'(x)$ has exactly one simple zero in each of the
intervals $\,\big(\pi/2,\,\pi),\,\,\big((k+1/2)\pi,\,(k+3/2)\pi\big).$

If $f(t)$ is convex at the same time, then $tf(t)$ is also convex and has the limiting value
$\,\lim_{t\to 0+} tf(t) =0\,$ and thus those zeros of $U'(x)$ lie in
$\,\big(k\pi,\,(k+1/2)\pi\big)\,$ according to Theorem A. Similarly, those zeros of $V'(x)$ lie in
$\,\big((k-1/2)\pi,\,k\pi\big)\,$ unless $\,(tf(t))'$ belongs to the exceptional case or equivalently
$f(t)$ belongs to the exceptional case.

On the other hand, as explained in the preceding section, it is simple to find that
$\,(-1)^k U(k\pi)>0\,$ for all $k$ when $f(t)$ is increasing and convex.
An immediate consequence of expansion formula \eqref{PW1} is that
\begin{equation}\label{Z1}
W\left[U(x), \,\frac{\sin x}{x}\right]>0\qquad (x>0).
\end{equation}
Let us denote by $(\sigma_k)$ all of the positive zeros of the derivative
$$ \frac{d}{dx} \left(\frac{\sin x}{x}\right) = \frac{x\cos x - \sin x}{x^2} $$
arranged in ascending order of magnitude. Since
$\,k\pi<\sigma_k<(k+1/2)\pi\,$ for each $k$, the Wronskian inequality \eqref{Z1}
implies that $\, (-1)^k U'(\sigma_k)<0.$ By combining with the above result, we may conclude that
$U'(x)$ has exactly one simple zero in each of the intervals $\,\big(k\pi, \,\sigma_k\big).$

What have been proved can be summarized as follows.

\begin{theorem}\label{theorem3}
Suppose that $f(t)$ is positive, increasing for $\,0<t<1\,$ and the integral
$\int_0^1 f(t) dt$ exists. Let $(\sigma_k)$ denote the sequence of all positive roots of the equation $\,\tan x = x\,$
arranged in ascending order of magnitude so that $\,k\pi<\sigma_k<(k+1/2)\pi\,$
for each $k$.

\begin{itemize}
\item[\rm(i)] $\,U'(x)<0\,$ for $\,0<x\le \pi\,$ and $U'(x)$ has exactly one simple zero
in each of the intervals $\,\big(k\pi,\,(k+1)\pi\big)\,$
with no positive zeros elsewhere; if, in addition, $f(t)$ is convex, then $U'(x)$ has a unique simple zero
in each of the intervals $\,\big(k\pi,\,\sigma_k\big),$ where $\,k=1,2,\cdots.$

\item[\rm(ii)] $\,V'(x)>0\,$ for $\,0<x\le \pi/2\,$ and $V'(x)$ has exactly one simple zero
in each of the intervals
$\,\big(\pi/2,\,\pi),\,\,\big((k+1/2)\pi,\,(k+3/2)\pi\big)\,$
with no positive zeros elsewhere; if, in addition, $f(t)$ is convex and in the general case, then $V'(x)$ has a unique simple zero
in each of the intervals
$\,\big((k-1/2)\pi,\,k\pi\big),$ where $\, k=1,2,\cdots.$
\end{itemize}
\end{theorem}

Due to the obvious relation
$$ U'(x) = -x \int_0^1\left[\int_t^1 s f(s) ds\right] \cos xt\,dt\qquad(x>0),$$
valid when $f(t)$ is continuous, part (i) yields the following.

\begin{corollary}\label{corollaryS}
If $f(t)$ is a positive increasing convex function for $\,0<t<1\,$ such that
the integral $\int_0^1 f(t) dt$ exists, then the Fourier transform
$$\fc\left[\int_t^1 s f(s) ds\right](x),\quad x>0,$$
is positive for $\,0<x\le\pi\,$ and has exactly one simple zero in each of the intervals
$\,\big(k\pi, \,\sigma_k\big),\, k=1,2,\cdots, $ with no positive zeros elsewhere.
\end{corollary}

In consideration of the Bessel function $J_\nu(x)$ represented by
\begin{equation}
J_\nu(x) = \frac{2(x/2)^\nu}{\sqrt{\pi}\,\Gamma(\nu+1/2)}\,\int_0^1 (1-t^2)^{\nu-1/2}
\cos xt\,dt,
\end{equation}
where $\,\nu>-1/2\,$ (\cite[3.3]{Wa}), it is simple to find that the
density function $\,f_\nu(t) = (1-t^2)^{\nu-1/2}\,$ has the recursive property
\begin{equation}\label{B1}
f_{\nu+1}(t) = (2\nu+1)\int_t^1 s f_\nu(s) ds,\quad 0<t<1.
\end{equation}

If $\,-1/2<\nu<1/2,$ being $f_\nu(t)$ increasing and convex for $\,0<t<1,$ it follows
from Theorem A that $J_\nu(x)$ has exactly one simple zero in each of the intervals
$\,\big((k-1/2)\pi,\,k\pi\big),\, k=1,2,\cdots,\,$ and no positive zeros elsewhere.

As customary, let $\big(j_{\nu, k}\big)$ be the sequence of all positive zeros of
$J_\nu(x)$ arranged in ascending order of magnitude. A classical theorem due to Schl\"afli
and Watson asserts that the map $\,\nu\mapsto j_{\nu, k},$ with $k$ fixed, is continuous and
strictly increasing for $\,\nu>-1\,$ (see \cite[15.6]{Wa}). Since
$$ J_{- 1/2}(x) = \sqrt{\frac{2}{\pi x}} \cos x,\quad
J_{1/2}(x) = \sqrt{\frac{2}{\pi x}} \sin x,$$
the Schl\"afli-Watson theorem implies that
$$\,j_{-1/2, k} =(k-1/2)\pi<j_{\nu, k}<k\pi = j_{1/2, k}\,,\quad k=1,2,\cdots,$$
for $\,-1/2<\nu<1/2,$ the same bounds obtained by Theorem A.  In this sense,
Theorem A gives the best possible lower and upper bounds for the $k$th positive zeros of $J_\nu(x)$
valid for the range $\,-1/2<\nu<1/2.$

If $\,1/2<\nu<3/2,$ then $f_\nu(t)$ is decreasing and concave for $\,0<t<1\,$ so that
$\,k\pi<j_{\nu, k}<(k+1)\pi\,$ for each $k$ according to Theorem B. Since
$$ J_{3/2}(x) = \sqrt{\frac{2}{\pi x}} \left(\frac{\sin x}{x} -\cos x\right),$$
which does not vanish at $\, x= (k+1)\pi,$ these upper bounds are not optimal.
By using the recursive property \eqref{B1}, however, an application of Corollary \ref{corollaryS}
to the case $\,-1/2<\nu<1/2\,$ shows that
$J_{\nu+1}(x)$ has exactly one simple zero in each of the intervals
$\,\big(k\pi,\,\sigma_k\big),$ which implies that
$\,k\pi<j_{\nu, k}<\sigma_k\,$ for $\,1/2<\nu<3/2.$ On noticing
$\,\sigma_k = j_{3/2, k}\,$ for each $k$, we find that $\,j_{\nu, k}\to \sigma_k\,$ as $\,\nu\to 3/2\,$
and thus the upper bounds $(\sigma_k)$ are optimal.

\begin{remark}
In \cite{KK}, Ki and Kim \cite{KK} found the following relations
\begin{align*}
&(x\sin x +\cos x) U(x) + x\cos x\,U'(x)= f(0)\nonumber\\
&\qquad\qquad + \int_0^1 \left[ 1-(\sin x \sin xt + t\cos x\cos xt)\right] df(t),\\
&(\sin x -x\cos x) V(x) + x\sin x\,V'(x)= (1-\cos x)f(0)\nonumber\\
&\qquad\qquad + \int_0^1 \left[ 1-(\cos x \cos xt + t\sin x\sin xt)\right] df(t),
\end{align*}
valid when $f(t)$ is defined for $\,0\le t<1\,$ and integrable.
If $f(t)$ is positive, increasing and in the general case, then it is a simple matter of algebra
to deduce from these relations the Wronskian inequalities
\begin{equation*}
W\left[\frac{\cos x}{x},\,U(x)\right]>0,\quad W\left[\frac{\sin x}{x},\,V(x)\right]>0
\qquad(x>0).
\end{equation*}
Proceeding in the same way as above, it is immediate to obtain an analogue
of Theorem \ref{theorem3} on the basis of these inequalities. Due to less precise characters,
however, we shall not present further details or applications.
\end{remark}

\section{Improvement of Theorem B}
This section focuses on obtaining an improvement of P\'olya's results, as stated in Theorem B,
on the distribution of zeros of Fourier cosine transforms for monotonically decreasing concave density functions.

As a motivational example, the Fourier cosine transform
\begin{align}\label{CV1}
\int_0^1 (1-t^2) \cos xt \,dt = \frac{\,2(\sin x - x\cos x)\,}{x^3}\qquad(x>0)
\end{align}
has simple zeros $(\sigma_k)$, where $\,k\pi<\sigma_k<(k+1/2)\pi\big)\,$ for each $k$.
We note that Theorem B is applicable because
the function $\,f(t) = 1- t^2\,$ is clearly positive, decreasing, concave for $\,0<t<1\,$
and its derivative $\,-f'(t) = 2t\,$ is in the general case. Nevertheless,
the only information we can deduce from Theorem B is that \eqref{CV1} has a unique simple zero
in each of the intervals $\,\big(k\pi,\,(k+1)\pi\big),\,k=1,2,\cdots,$ which is far from being optimal.

Our improvement of Theorem B reads as follows.

\begin{theorem}\label{theorem4}
For a positive differentiable function $f(t)$ for $\,0<t<1,$ assume that
$-f'(t)$ is positive, increasing, convex, in the general case and
integrable. Then $U(x)$ is positive for $\,0<x\le\pi\,$ and has an
infinity of zeros for $\,x>\pi\,$ all of which are simple and distributed as follows, where
$$ L = \lim_{t\to 0+} \left[- f'(t)\right],\quad M = \lim_{t\to 1-} f(t).$$

\begin{itemize}
\item[\rm(i)] For any $\,L\ge 0,\,M\ge 0,$ each of the intervals
\begin{equation*}
\big((2k-1)\pi,\,2k\pi\big),\,\,\,\big(2k\pi,\,(2k+1/2)\pi\big),\quad k=1,2,\cdots,
\end{equation*}
contains exactly one zero and $U(x)$ has no positive zeros elsewhere.
\item[\rm(ii)] If either $\,L = 0\,$ or $\,0<2L\le 3\pi M,$ then each of the intervals
$$\big(k\pi, \,(k+1/2)\pi\big),\quad k=1,2,\cdots,$$
contains exactly one zero and $U(x)$ has no positive zeros elsewhere.
\end{itemize}
\end{theorem}

\begin{proof} The assertion that $\,U(x)>0\,$ for $\,0<x\le\pi\,$ is proved in Lemma \ref{lemma2}.

Our proof for the distribution of zeros is based on the formula
\begin{align}\label{CV2}
\left\{\begin{aligned}
&{\,U(x) = W(x) + L\left(\frac{1-\cos x}{x^2}\right) + M\left( \frac{\sin x}{x}\right),}\\
&{W(x) \equiv \frac 1x \int_0^1\left[-f'(t) -L\right] \sin xt\,dt.}\end{aligned}\right.
\end{align}
To see this, we integrate by parts to write $U(x)$ into the form
\begin{align*}
U(x) &= \int_0^1 [f(t)- M]\cos xt\,dt + M\int_0^1 \cos xt\,dt\\
&=\frac{1}{x}\int_0^1\left[-f'(t)\right] \sin xt\,dt + M\left(\frac{\sin x}{x}\right).
\end{align*}
On decomposing $\,-f'(t) = [-f'(t) -L] +L,$ separating the first integral accordingly
and evaluating the latter part, formula \eqref{CV2} follows at once.

By the assumption, the function $\,g(t) = -f'(t) -L\,$
is positive, increasing, convex for $\,0<t<1\,$ and $\,g(0+) =0.$
Moreover, $g(t)$ is in the general case and the integral $\int_0^1 g(t) dt $ exists.
It follows from P\'olya's results that
\begin{equation}\label{CV3}
(-1)^k W(k\pi)<0, \quad (-1)^k W\left[\left(k+1/2\right)\pi\right]>0,
\end{equation}
where $\, k=1,2, \cdots\,$ (see \cite[section 6]{P} for details).

Evaluating at $\,x=k\pi,$ the formula \eqref{CV2} gives
$$ U(k\pi) = W(k\pi)+ \left\{
\begin{aligned}
&{\frac{2L}{k^2 \pi^2}\quad\,\,\text{if $\,k\,$ is odd},}\\
&{\quad 0\qquad\text{if $\,k\,$ is even}}\end{aligned}\right.$$
and the first inequality of \eqref{CV3} indicates that $\,(-1)^k U(k\pi)<0\,$ for each $k$.
Evaluating at $\,x=(2k+1/2)\pi,$ we also find that
\begin{align*}
U\left[\left(2k+1/2\right)\pi\right] &= W\left[\left(2k+1/2\right)\pi\right]\\
&\quad + \frac{L}{[(2k+1/2)\pi]^2} + \frac{M}{(2k+1/2)\pi}>0
\end{align*}
by the second inequality of \eqref{CV3}. Consequently, it follows from
Theorem \ref{theoremP} and Remark \ref{remarkP} that $U(x)$ has one and only one simple zero
inside each of the intervals described in part (i)
and has no zeros elsewhere, no matter what the values of $L, M$ are (note that
$\,L\ge 0,\,M\ge 0\,$ due to the assumption).

If $\,L =0,$ the second inequality of \eqref{CV3} implies that
\begin{align*}
U\left[\left(2k-1/2\right)\pi\right] =
W\left[\left(2k-1/2\right)\pi\right]
-\frac{M}{(2k-1/2)\pi}<0
\end{align*}
for each $k$ and hence $U(x)$ has a unique zero in each of the intervals described in part (ii).
In the case $\,0<2L\le 3\pi M,$ \eqref{CV3} implies that
\begin{align*}
U\left[\left(2k-1/2\right)\pi\right] = W\left[\left(2k-1/2\right)\pi\right]
+\frac{\,2L -(4k-1)\pi M\,}{2[(2k-1/2)\pi]^2} <0
\end{align*}
for each $k$, due to $\,2L-(4k-1)\pi M\le 2L - 3\pi M\le 0,\,$ which leads to the same
conclusion as above. Our proof is now complete.
\end{proof}

In practice, if $f(t)$ is a non-linear function satisfying
\begin{equation}\label{CV4}
f(t)>0,\quad f'(t)<0,\quad f''(t)\le 0, \quad f^{(3)}(t)\le 0
\end{equation}
for $\,0<t<1,$ as well as the integrability condition for $f'(t)$, then it is trivial to see that all requirements of
the above theorem are fulfilled. For example, as the function $\,f(t) = a - bt^2,\,0<b\le a,$ satisfies \eqref{CV4}
with $\,f'(0) =0,$ it follows from Theorem \ref{theorem4} that its Fourier cosine transform
\begin{align}
\int_0^1 f(t)\cos xt \,dt =\frac{\,[(a-b) x^2 + 2b] \sin x - 2bx \cos x\,}{x^3}
\end{align}
has exactly one simple zero in each of the intervals
$$\,\big(k\pi,\,(k+1/2)\pi\big),\quad k =1, 2, \cdots,$$
which can be readily verified by an elementary inspection.

\begin{corollary}\label{corollary2}
Let $f(t)$ be a positive, increasing and convex function for $\,0<t<1\,$ such that
it is in the general case and the integral $\int_0^1 f(t) dt$ exists.
Then $V(x)$ or the Fourier cosine transform related by
$$\fc\left[\int_t^1 f(s) ds\right](x) = \frac{1}{x}\,V(x)\qquad(x>0)$$
is positive for $\,0<x\le\pi\,$ and has an
infinity of zeros for $\,x>\pi\,$ all of which are simple and distributed as follows.

\begin{itemize}
\item[\rm(i)] If $\,f(0+) =0,$ then each of the intervals
$$\big(k\pi, (k+1/2)\pi\big),\quad k=1,2,\cdots,$$
contains exactly one zero and $V(x)$ has no positive zeros elsewhere.
\item[\rm(ii)] If $\,f(0+)>0,$ then each of the intervals
\begin{equation*}
\big((2k-1)\pi,\,2k\pi\big),\,\,\big(2k\pi,\,(2k+1/2)\pi\big),\quad k=1,2,\cdots,
\end{equation*}
contains exactly one zero and $V(x)$ has no positive zeros elsewhere.
\end{itemize}
\end{corollary}

\begin{remark}
While part (i) is merely a restatement of P\'olya's results, as stated in (A2), Theorem A,
part (ii) coincides with Sedletskii's result \cite{Se} for which the author made use of convexity
in his proof.
\end{remark}

As an illustration, the Struve function represented by
\begin{align}
\bH_\nu(x) = \frac{2(x/2)^\nu}{\sqrt{\pi}\,\Gamma\left(\nu +1/2\right)}
\int_0^1 (1-t^2)^{\nu-1/2}\sin(xt) \,dt\label{St2}
\end{align}
has exactly one positive simple zero in each of the intervals
\begin{equation*}
\big(\pi, \,2\pi\big),\,\,\left(2\pi,\,\frac{5\pi}{2}\right),\,\,
\big(3\pi, \,4\pi\big),\,\,\left(4\pi,\,\frac{9\pi}{2}\right),\cdots
\end{equation*}
when $\,|\nu|<1/2.$  We refer to Steinig \cite{St} for more precise bounds of zeros.

\section{Remarks on Kuttner's problem}

As one of the long-standing open problems in classical analysis,
Kuttner's problem concerns the range of parameters $\,\delta>0,\,\lambda>0\,$ for which
\begin{equation}\label{K1}
\int_0^1 \left(1- t^\delta\right)^\lambda \cos xt\,dt\ge 0\qquad(x>0).
\end{equation}
While we refer to \cite{Gn}, \cite{Ku}, \cite{MR} for some of the known results, methods and generalizations,
we are interested in those outcomes obtainable from P\'olya's criteria or our improvements.
Of particular concern is the distribution of zeros in the case when \eqref{K1} fails,
which appears to be less known.

In the first place, as the function $\,f(t) = (1-t^\delta)^\lambda\,$ has derivatives
\begin{align*}
f'(t) &= -\delta\lambda (1-t^\delta)^{\lambda-1} t^{\delta-1},\\
f''(t) &= -\delta\lambda (1-t^\delta)^{\lambda-2} t^{\delta-2}\left[\delta-1-(\delta\lambda-1)t^\delta\right],
\end{align*}
it is elementary to find that $f(t)$ is decreasing for any $\,\delta>0,\,\lambda>0\,$ and that
it is convex for $\,0<\delta\le 1\le\lambda\,$ and concave for $\,0<\lambda\le 1\le \delta.$
Since $\,f(1-) =0\,$ and $f'(t)$ is in the general case unless $\,\delta =\lambda =1,$
Corollary \ref{corollary1} and Theorem B yield the following results immediately.

\begin{proposition}\label{proposition1}
For $\,\delta>0,\,\lambda>0,$ let
\begin{equation}\label{K2}
\Omega(x) = \int_0^1 \left(1- t^\delta\right)^\lambda \cos xt\,dt\qquad(x>0).
\end{equation}

\begin{itemize}
\item[\rm(i)] If $\,0<\delta\le 1\le\lambda,$ then $\,\Omega(x)\ge 0\,$ for all $\,x>0\,$ and strict
inequality holds true unless $\,\delta=\lambda =1.$
\item[\rm(ii)] If $\,0<\lambda\le 1\le \delta,\, (\delta, \lambda)\ne (1,1),$ then $\,\Omega(x)>0\,$
for $\,0<x\le\pi\,$ and has one and only one simple zero in each of the intervals
$$\big(k\pi,\,(k+1)\pi\big),\quad k=1,2,\cdots,$$
with no positive zeros elsewhere.
\end{itemize}
\end{proposition}

\begin{figure}[!ht]
 \centering
 \includegraphics[width=320pt, height=300pt]{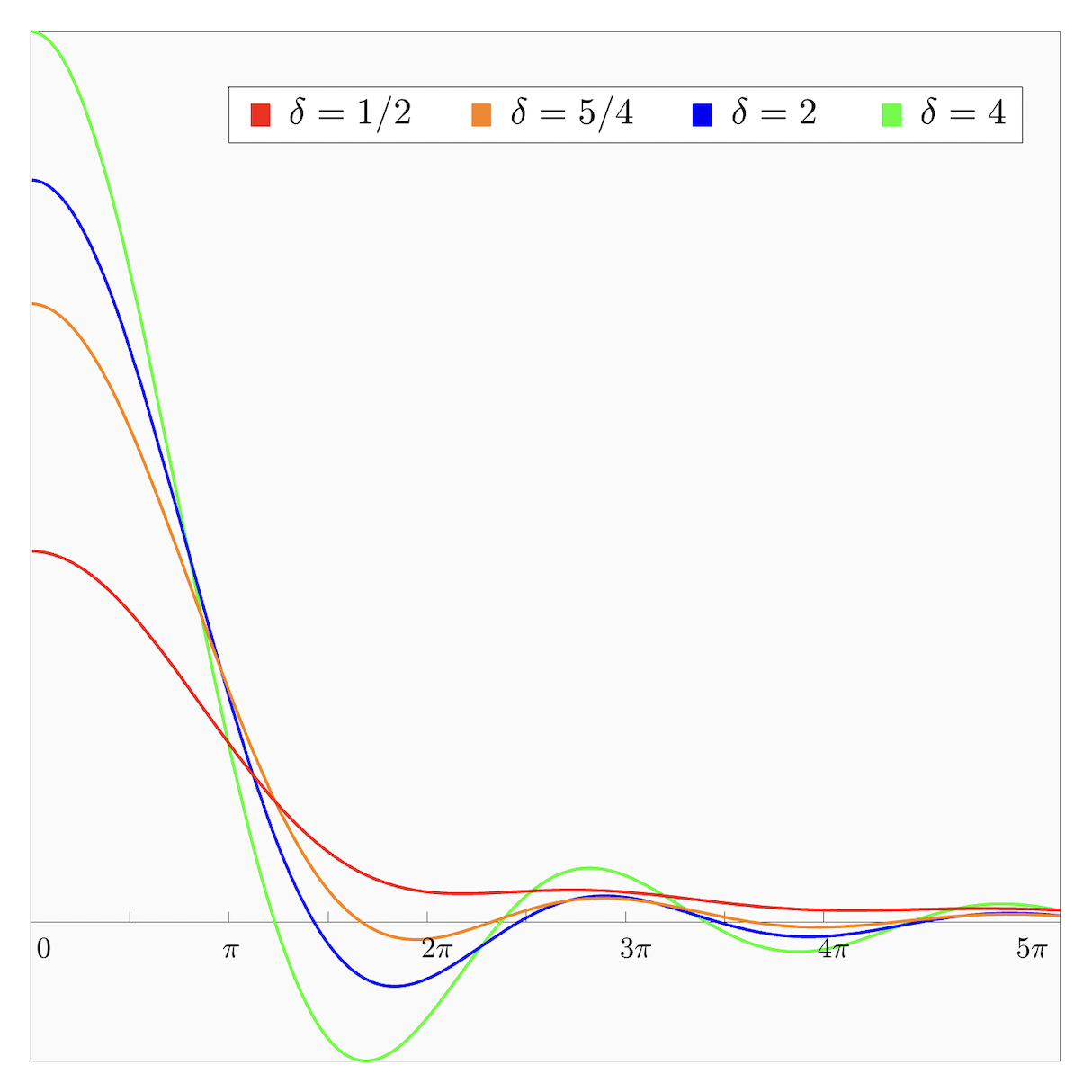}
 \caption{The graphs of $\,\Omega(x) = \int_0^1 \left(1- t^\delta\right)\cos xt\,dt\,$
 in the first quadrant. While the graph in the case $\,\delta =1/2\,$ lies above the $x$-axis,
 the graphs in other cases are oscillatory. In the particular case $\,\delta =2,$
 the zeros of $\Omega(x)$ coincide with $(\sigma_k)$, the positive roots of $\,\tan x = x.$ }
\label{Fig2}
\end{figure}

We remark that part (i) is known but part (ii) seems to be unavailable in the literature.
In the special case $\,\lambda =1,\,\delta\ge 2, $ it is still possible to improve the bounds of zeros further
in the following manner (see Figure \ref{Fig2}).

\begin{itemize}
\item If $\,\lambda =1,\,\delta\ge 2,$ then we have
$$f^{(3)}(t) = -\delta(\delta-1)(\delta-2) t^{\delta-3}\le 0,\quad f'(0+) = 0$$
so that we may apply Theorem \ref{theorem4} to find that $\Omega(x)$ has a unique zero
in each of the intervals $\,\big(k\pi,\,(k+1/2)\pi\big).$

\item If $\,\lambda =1,\,\delta>3,$ then we write
$$ 1- t^\delta = \int_t^1 s h(s) ds, \quad\text{where}\quad  h(s) = \delta s^{\delta-2},$$
and note that $h(s)$ is positive, increasing, convex for $\,0<t<1\,$ and the integral $\int_0^1 h(s)ds$ exists.
By an application of Corollary \ref{corollaryS}, we find that $\Omega(x)$ has a unique zero
in each of the intervals $\,\big(k\pi,\,\sigma_k\big).$
\end{itemize}

\begin{proposition}
If $\,\lambda =1,\,\delta>1,$ then $\Omega(x)$ has exactly one simple zero
in each of the following intervals and no positive zeros elsewhere:
\begin{align*}
&\,\,{\rm(i)} \,\,\,\, \big(k\pi,\,(k+1)\pi\big)\quad\text{for}\quad 1<\delta<2,\\
&\,{\rm(ii)}\,\,\,\,\big(k\pi,\,(k+1/2)\pi\big)\quad\text{for}\quad 2\le\delta\le 3,\\
&{\rm(iii)}\,\,\,\big(k\pi,\,\sigma_k\big)\quad\text{for}\quad \delta>3,\quad\text{where}\quad k=1,2,\cdots.
\end{align*}
\end{proposition}

\section{Main applications}
This section deals with the distribution of zeros
for the Fourier transform of the beta probability density function defined in \eqref{B0}, that is,
\begin{equation*}
f(t)= \frac{1}{B(\alpha, \beta)}\, (1-t)^{\alpha-1} t^{\beta-1},\quad 0<t<1.
\end{equation*}

For $\,x>0,$ by expanding $\cos xt$ or $\sin xt$ into ${}_0F_1$ hypergeometric series and integrating termwise,
it is elementary to represent (see \cite[13.1 (56)]{EMOT})
\begin{align}
\Phi(x) &= \int_0^1 f(t) \cos xt\,dt\nonumber\\
&= {}_2F_3\left[\begin{array}{c}
\beta/2, \,\,\,(\beta+1)/2\\1/2, \,(\alpha +\beta)/2,\,
(\alpha+\beta+1)/2\end{array}\biggr| -\frac{x^2}{4}\right],\label{B01}\\
\Psi(x) &= \int_0^1 f(t) \sin xt\,dt = \frac{\beta x}{\alpha+\beta} \nonumber\\
&\times  {}_2F_3\left[\begin{array}{c}
(\beta+1)/2, \,\,\,(\beta+2)/2\\3/2, \,(\alpha +\beta+1)/2,\,
(\alpha+\beta+2)/2\end{array}\biggr| -\frac{x^2}{4}\right].\label{B02}
\end{align}

In the recent work \cite{CC}, Cho and Chung investigated the positivity problem for the general
${}_2F_3$ hypergeometric function of the form
\begin{equation*}
{}_2F_3\left[\begin{array}{c}
a_1, a_2\\ b_1, b_2, b_3\end{array}\biggr| -x^2\right]\qquad(x>0),
\end{equation*}
where all of parameters are assumed to be positive, and obtained a list of sufficient conditions
in terms of Newton polyhedra with vertices
consisting of permutations of $\,(a_2, a_1+1/2, 2a_1)\,$ or $\,(a_1, a_2+1/2, 2a_2).$

\begin{theorem}\label{theoremCC1}
Define $\,\mathcal{P}_c\subset\mathcal{P}_s\subset\R_+^2\,$ by
\begin{align*}
\mathcal{P}_c &= \left\{(\alpha, \beta): \alpha\ge \frac 53,\,\,0<\beta\le
\min\big(1, \,\,\alpha -1\big),\,\,(\alpha, \beta)\ne (2, 1)\right\}\\
&\qquad\cup \,\left\{(\alpha, \beta): 1\le\alpha\le \frac 53,\,\,0<\beta\le \frac 23\right\},\\
\mathcal{P}_s &= \biggl\{(\alpha, \beta): \alpha\ge \frac 12,\,\,
0<\beta\le\min\biggl(2, \,\,\frac{\alpha+1}{2}, \,\,2\alpha-1\biggr),\\
&\qquad\qquad\text{and}\quad (\alpha, \beta)\ne (1, 1)\biggr\}.
\end{align*}

\begin{itemize}
\item[\rm(i)] If $\,(\alpha, \beta)\in\mathcal{P}_c,$ then $\,\Phi(x)>0\,$
for all $\,x>0.$

\item[\rm(ii)] If $\,(\alpha, \beta)\in\mathcal{P}_s,$ then $\,\Psi(x)>0\,$
for all $\,x>0.$

\item[\rm(iii)] If $\,\beta>\alpha,$ then both $\Phi(x), \Psi(x)$ change sign infinitely often.
Moreover, $\Phi(x)$ changes sign at least once when $\,\beta\le\alpha,\,\beta>1\,$
and $\Psi(x)$ changes sign at least once when $\,\beta\le\alpha,\,\beta>2.$
\end{itemize}
\end{theorem}

\begin{proof}
All stated results are part of \cite[Theorem 9.1, (S5)]{CC} except that the
region of positivity for $\Phi(x)$ in part (i) is enlarged slightly. Indeed, it
is shown therein that $\,\Phi(x)>0\,$ for all $\,x>0\,$ when
$\,(\alpha, \beta)\in\Delta_c\subset\mathcal{P}_c$, where
$$\Delta_c = \Big\{(\alpha, \beta): \alpha>1,\,\,0<\beta\le
\min\big(1, \,\,\alpha -1\big),\,\,(\alpha, \beta)\ne (2, 1)\Big\}.$$

In the case $\,\alpha =1,$ it is pointed out in the example \eqref{G2} that
$\,\Phi(x)>0\,$ for all $\,x>0\,$ when $\,0<\beta\le 2/3.$ Setting
$\,\Phi(x) \equiv \Phi(\alpha, \beta; x),$ it is easy to see from the transference principle
established in \cite[Proposition 2.1]{CC} that if $\,\Phi(\alpha_0, \beta_0; x)\ge 0\,$ for all $\,x>0,$
then $\,\Phi(\alpha, \beta_0; x)>0\,$ for all $\,x>0\,$ with any $\,\alpha>\alpha_0.$
Consequently, $\,\Phi(x)>0\,$ for all $\,x>0\,$ when
$\,(\alpha, \beta)\in\mathcal{R}_c$, where
$\,\mathcal{R}_c = \big[1, \,\infty\big)\times \big(0,\,2/3\big],$
and since $\,\mathcal{P}_c = \Delta_c \cup \mathcal{R}_c,$ part (i) follows.
\end{proof}

\begin{remark}
More extensively, it is shown that the Fourier sine transform $\Psi(x)$ is also
positive for $\,x>0\,$ in the range
$$\alpha\ge 1/2,\quad -1<\beta\le \min\big(0, \,\,\alpha-1\big).$$
\end{remark}

We note that both Fourier transforms $\Phi(x), \Psi(x)$ are strictly positive for $\,x>0\,$
when $\,(\alpha, \beta)\in\mathcal{P}_c.$ On making use of the symmetry
\begin{equation}\label{B2}
f(\alpha, \beta; 1-t) = f(\beta, \alpha; t),\quad 0<t<1,
\end{equation}
where we temporarily put $\,f(t) = f(\alpha, \beta; t)\,$ to indicate order of parameters,
Corollary \ref{corollaryP} and Theorem \ref{theoremCC1} yield the following results at once,
what we aimed to establish in the present work.

\begin{theorem}\label{theoremCC2}
Define $\,\mathcal{P}_c^*\subset \mathcal{P}_s^*\subset \R_+^2\,$ by
\begin{align*}
\mathcal{P}_c^* &= \left\{(\alpha, \beta): 0<\alpha\le 1,\,\,\beta\ge \max \left(\frac 53,\,\, \alpha+1\right),
\,\,(\alpha, \beta)\ne (1, 2)\right\}\\
&\qquad\cup \,\left\{(\alpha, \beta): 0<\alpha\le \frac 23,\,\,1\le\beta\le \frac 53\right\},\\
\mathcal{P}_s^* &= \biggl\{(\alpha, \beta): 0<\alpha\le 2,\,\,
\beta\ge\max\biggl(\frac 12, \,\,\frac{\alpha+1}{2}, \,\,2\alpha-1\biggr)\nonumber\\
&\qquad\qquad\text{and}\quad (\alpha, \beta)\ne (1, 1)\biggr\}.
\end{align*}

\begin{itemize}
\item[\rm(i)] If $\,(\alpha, \beta)\in \mathcal{P}_c^*,$ then
$\Phi(x)$ has exactly one simple zero in each of the intervals
$\,\big((k-1/2)\pi,\,k\pi\big)\,$ and so does $\Psi(x)$
in each of the intervals
$\,\big(k\pi,\,(k+1/2)\pi\big),$ where $\,k=1,2,\cdots.$
Moreover, $\Phi(x), \Psi(x)$ have no positive zeros elsewhere and share no common zeros.

\item[\rm(ii)] If $\,(\alpha, \beta)\in \mathcal{P}_s^*\setminus \mathcal{P}_c^*,$ then
$\Phi(x)$ has exactly one simple zero in each of the intervals
$\,\big((k-1/2)\pi,\,(k+1/2)\pi\big)\,$ and so does $\Psi(x)$
in each of the intervals
$\,\big(k\pi,\,(k+1)\pi\big),$ where $\,k=1,2,\cdots.$
Moreover, $\Phi(x), \Psi(x)$ have no positive zeros elsewhere and share no common zeros.
\end{itemize}
\end{theorem}

\begin{figure}[!ht]
 \centering
 \includegraphics[width=300pt, height=300pt]{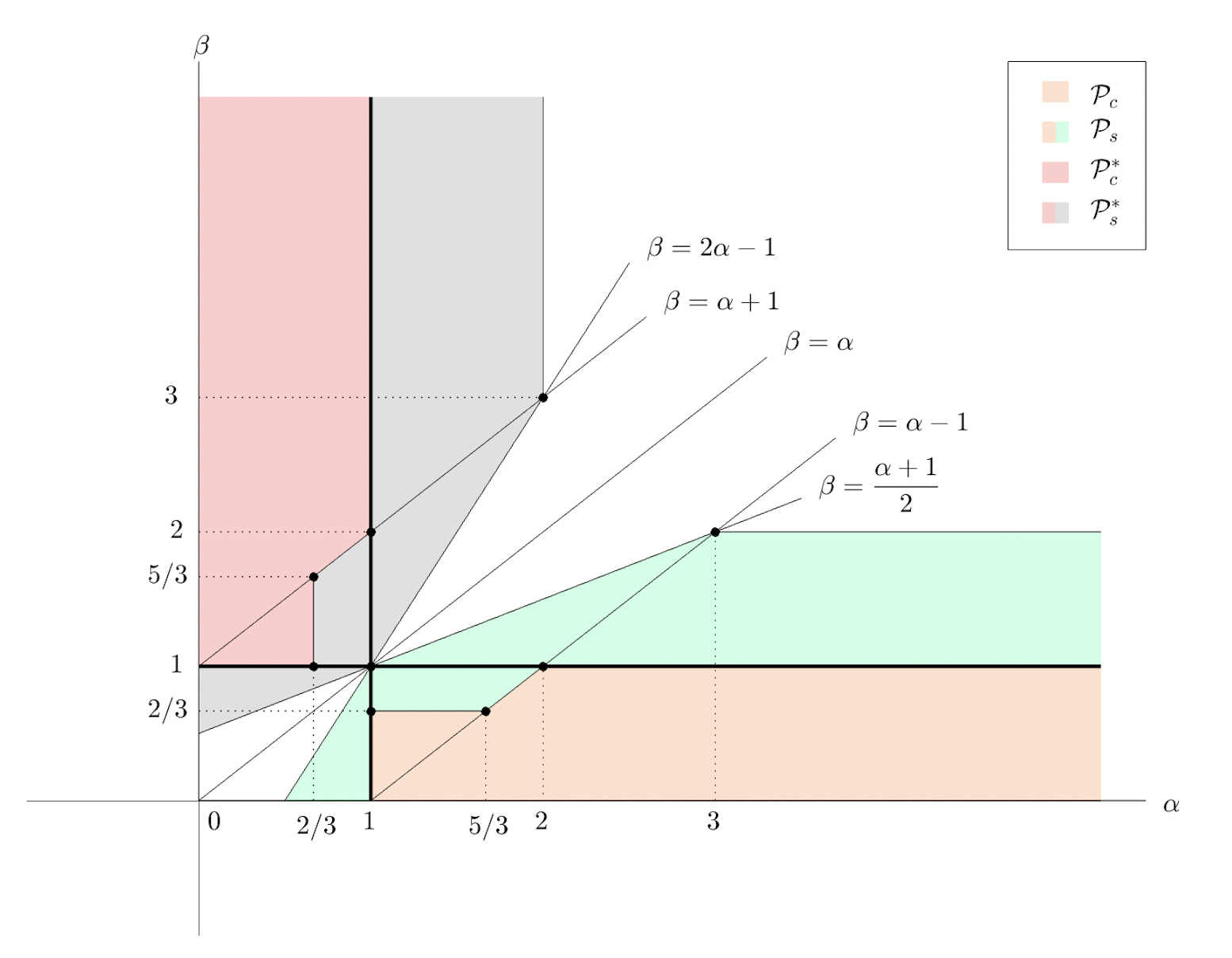}
 \caption{$\mathcal{P}_c, \mathcal{P}_s$ represent the regions of positivity and
 $\mathcal{P}_c^*, \mathcal{P}_s^*$ represent their reflections about the line $\,\beta =\alpha.$ The union of
 horizontal and vertical strips with corner at $(1,1)$, the region bounded by black-thick lines and coordinate axes,
 represents P\'olya's region of monotonicity.}
\label{Fig1}
\end{figure}

It should be emphasized that $\,\mathcal{P}_c^*, \mathcal{P}_s^*\,$
arise as the reflections of $\,\mathcal{P}_c, \mathcal{P}_s\,$ across the line $\,\alpha =\beta,$ respectively.
As shown in Figure \ref{Fig1}, each of the sets
$\,\mathcal{P}_c, \mathcal{P}_s, \,\mathcal{P}_c^*, \mathcal{P}_s^*\,$
represents an infinite polygonal region in the first quadrant.

\begin{itemize}
\item[(1)] On inspecting the first and second derivatives,
it is elementary to find that the beta density function
$f(t)$ is decreasing when $\,(\alpha, \beta)\in \mathcal{D},$
decreasing and convex when $\,(\alpha, \beta)\in \mathcal{C},$ where
\begin{align*}
\mathcal{C} &= \big\{(\alpha, \beta) : \alpha\ge 2, \,0<\beta<1\quad\text{or}
\quad \alpha>2, \,\beta =1\big\}\\
&\qquad\cup\,\big\{(\alpha, \beta) : \alpha =1, \,0<\beta<1\big\},\\
\mathcal{D} &= \big\{(\alpha, \beta) : \alpha\ge 1,\,0<\beta<1\quad\text{or}
\quad \alpha>1,\,\beta =1\big\}.
\end{align*}

Due to symmetry \eqref{B2}, since the composition with $\,t\mapsto 1-t\,$ changes monotonicity
but preserves convexity, we find that $f(t)$ is increasing when $\,(\alpha, \beta)\in \mathcal{D}^*,$
increasing and convex when  $\,(\alpha, \beta)\in \mathcal{C}^*,$ where
$\mathcal{C}^*, \mathcal{D}^*$ denote the reflections of $\mathcal{C}, \mathcal{D}$ about the line
$\,\alpha =\beta.$

In addition, it is easy to see that $f(t)$ is decreasing and concave only when
$\,1<\alpha<2,\,\beta =1.$
As readily observed, the above two theorems contain all of those results obtainable from
P\'olya's criteria except for the result that
\emph{if $\,1<\alpha<2,\,\beta =1,$ then
$\Phi(x)$ has exactly one simple zero in each of the intervals
$\,\big(k\pi,\,(k+1)\pi\big),\, k=1,2,\cdots.$}

\item[(2)] As Figure \ref{Fig1} clearly indicates, both theorems still leave many cases
of parameters unanswered on the distribution of zeros or positivity.
If $\,\beta>\alpha,\,(\alpha, \beta)\notin\mathcal{P}_s^*,$ for example,
the third part of Theorem \ref{theoremCC1} tells that both $\Phi(x), \Psi(x)$ have
an infinity of zeros but any information on the distribution of zeros is not available yet.
In the diagonal case $\,\alpha =\beta,$ however, some partial results are obtainable. In fact, since
$\,f(1-t) = f(t)\,$ in such a case, it is simple to find that
\begin{align*}
\Phi(x) &= \cos x\, \Phi(x) + \sin x\, \Psi(x),\\
\Psi(x) &= \sin x\, \Phi(x) -\cos x\, \Psi(x),
\end{align*}
which immediately imply that \emph{if $\,\alpha =\beta,$ then
$$\Phi\big((2k-1)\pi\big) =0,\quad \Psi(2k\pi) =0,\quad k=1,2,\cdots.$$}
We do not know whether $\Phi(x), \Psi(x)$ have other positive zeros.

\item[(3)] In the special case $\,\beta =1,$ \eqref{B01}, \eqref{B02} reduce to
\begin{equation}
\begin{aligned}
&{s_{\alpha-3/2, 1/2} (x) = \frac{x^{\alpha-1/2}}{\alpha-1} \int_0^1 (1-t)^{\alpha-1} \cos xt \,dt,}\\
&{s_{\alpha-1/2, 1/2} (x) = x^{\alpha-1/2} \int_0^1 (1-t)^{\alpha-1} \sin xt \,dt,}
\end{aligned}
\end{equation}
where it must be assumed that $\,\alpha \ne 1\,$ in the fist formula
and $s_{\mu, \nu}(x)$ stands for the Lommel function of the first kind
(see \cite[10.7]{Wa}).
By applying both of the above theorems and combining in an obvious way, we obtain the following
properties of Lommel functions $s_{\mu, 1/2}(x)$ in the range $\,\mu>- 3/2,\,\mu\ne \pm 1/2, $ where $\, k=1,2,\cdots.$
\emph{
\begin{itemize}
\item It has a unique simple zero in each of the following intervals
$$
\left\{\begin{aligned}
&{\big(\left(k- 1/2\right)\pi,\,k\pi\big)\quad\text{if}\quad -3/2<\mu\le - 5/6,}\\
&{\big(\left(k- 1/2\right)\pi,\,\left(k+ 1/2\right)\pi\big)\quad\text{if}\quad -5/6<\mu<- 1/2,}\\
&{\big(k\pi,\,\left(k+ 1/2\right)\pi\big)\quad\text{if}\quad - 1/2<\mu\le 1/6,}\\
&{\big(k\pi,\,(k+1)\pi\big)\quad\text{if}\quad  1/6<\mu< 1/2,}
\end{aligned}\right.
$$
and no other positive zeros.
\item It is strictly positive for $\,x>0\,$ if $\,\mu > 1/2.$
\end{itemize}}
We refer to \cite{KL}, \cite{Ko} for sharper bounds of zeros, except the case $\,- 1/2<\mu\le 1/6,$
as well as its applications to trigonometric sums and \cite{CC1}, \cite{St2} for results in other parameter cases.

\item[(4)] In the special case $\,\beta =2,$ \eqref{B02} takes the form
\begin{equation}
\Psi(x) = \frac{1}{\alpha(\alpha+1)} \int_0^1 (1-t)^{\alpha-1} t \sin xt\,dt \equiv \Psi_\alpha(x),
\end{equation}
which was investigated by Williamson \cite{Wi} for the purpose of discerning whether
Laplace transforms of $\alpha$-monotone functions are univalent or not in the right-half plane.
Our results read as follows.
\emph{
\begin{itemize}
\item $\Psi_\alpha(x)$ has a unique simple zero in each of the following intervals
$$
\left\{\begin{aligned}
&{\big(k\pi,\,(k+1/2)\pi\big)\quad\text{if}\quad 0<\alpha\le 1,}\\
&{\big(k\pi,\,(k+1)\pi\big)\quad\text{if}\quad 1<\alpha\le 3/2,}\end{aligned}\right.
$$
and no other positive zeros.
\item $\Psi_\alpha(x)$ changes sign at least once for $\,x>0\,$ if $\,3/2<\alpha<3.$
\item $\Psi_\alpha(x)$ is strictly positive for $\,x>0\,$ if $\, \alpha\ge 3.$
\end{itemize}}
We should remark that these results not only extend Williamson's positivity result for $\,\alpha =3\,$ but also
disprove his conjecture that there exists $\,\alpha'\,$ with $\,2<\alpha'<3\,$ such that
$\Psi_\alpha(x)$ remains nonnegative for $\,\alpha\ge\alpha'\,$ but changes sign for $\,0<\alpha<\alpha'\,$
(see also \cite{CC}).

\item[(5)]
In his study \cite{Sb} on the uncertainty principle of Fourier transforms, Steinerberger introduced
the sequence $(a_k)$ defined by
\begin{equation}
a_k = \F\left[\begin{array}{c}
(1+\beta)/2\\ 3/2, \,(3+\beta)/2\end{array}\biggr| - \frac{(2k-1)^2 \pi^2}{16}\right]
\end{equation}
and raised an open question on the range of parameter $\beta$ for which $(a_k)$ alternates in sign.
On recognizing
$\,a_k = S_\beta\big[(k-1/2)\pi\big],$ where
\begin{align}
S_\beta(x) &=  \F\left[\begin{array}{c}
(1+\beta)/2\\ 3/2, \,(3+\beta)/2\end{array}\biggr| - \frac{x^2}{4}\right]\nonumber\\
&= \frac{1+\beta}{x}\int_0^1 t^{\beta-1} \sin xt\,dt,
\end{align}
which corresponds to the case $\,\alpha =1\,$ of \eqref{B02}, it is plain to deduce
from both theorems and the example \eqref{G2} the following answers.

\begin{proposition}
If $\,\beta\ge 2,$ then $\,a_k>0\,$ for each odd positive integer $k$
and $\,a_k<0\,$ for each even positive integer $k$. If $\,-1<\beta\le 5/3,$ then $\,a_k>0\,$ for
every positive integer $k$.
\end{proposition}

While the case $\,5/3<\beta<2\,$ is still inconclusive, it should be pointed out
that Steinerberg himself proved the case when
$\,\beta\in\big\{ 2, 3, 4, 5, 6\big\}\,$ by elementary computations.

\end{itemize}

\bigskip

\textsc{Acknowledgements}. We are grateful to Seok-Young Chung for his help
in translating P\'olya's original paper into English and for sharing his insightful ideas on the present project with us. 
Yong-Kum Cho is supported by
the National Research Foundation of Korea funded by 
the Ministry of Science and ICT (2021R1A2C1007437). Young Woong Park
is supported by
the Chung-Ang University Graduate Research Scholarship in 2021.


\begin{thebibliography}{12}

\bibitem[1]{CC1} Y.-K. Cho and S.-Y. Chung,
\emph{On the positivity and zeros of Lommel functions: Hyperbolic extension and interlacing},
J. Math. Anal. Appl. 470, pp. 898--910 (2019)

\bibitem[2]{CC} Y.-K. Cho and S.-Y. Chung,
\emph{The Newton polyhedron and positivity of ${}_2F_3$ hypergeometric functions},
Constr. Approx. (2021), https://doi.org/10.1007/s00365-021-09540-7

\bibitem[3]{CCY} Y.-K. Cho, S.-Y. Chung and H. Yun,
\emph{Rational extension of the Newton diagram for the positivity of ${}_1F_2$ hypergeometric functions
and Askey-Szeg\"o problem},
Constr. Approx. 51, pp. 49--72 (2020)

\bibitem[4]{CY} Y.-K. Cho and H. Yun,
\emph{Newton diagram of positivity for ${}_1F_2$ generalized hypergeometric functions},
Integr. Transf. Spec. F. 29, pp. 527--542 (2018)

\bibitem[5]{DR} D. K. Dimitrov and P. K. Rusev,
\emph{Zeros of entire Fourier transforms},
East J. Approx. 17, pp. 1--108 (2011)

\bibitem[6]{EMOT} A. Erd\'elyi, W. Magnus, F. Oberhettinger and F. G. Tricomi,
\emph{Tables of Integral Transforms, Vol. II}, McGraw-Hill (1954)


\bibitem[7]{G} G. Gasper,
\emph{Positive integrals of Bessel functions},
SIAM J. Math. Anal. 6, pp. 868--881 (1975)

\bibitem[8]{Gn} T. Gneiting,
\emph{Kuttner's problem and a P\'olya type criterion for characteristic functions},
Proc. Amer. Math. Soc. 128, pp. 1721--1728 (1999)

\bibitem[9]{KK} H. Ki and Y.-O. Kim,
\emph{On the number of nonreal zeros of real entire functions and the
Fourier-P\'olya conjecture},
Duke Math. J. 104, pp. 45--73 (2000)


\bibitem[10]{KL} S. Koumandos and M. Lamprecht,
\emph{The zeros of certain Lommel functions},
Proc. Amer. Math. Soc. 140, pp. 3091--3100 (2012)


\bibitem[11]{Ko} S. Koumandos,
\emph{Positive trigonometric integrals associated with some Lommel functions of the first kind},
Mediterr. J. Math. 14:15 (2017)



\bibitem[12]{Ku} B. Kuttner,
\emph{On the Riesz means of a Fourier series (II)},
J. London Math. Soc. 19, pp. 77--84 (1944)



\bibitem[13]{MR} J. Misiewicz and D. Richards,
\emph{Positivity of integrals of Bessel functions},
SIAM J. Math. Anal. 25, pp. 596--601 (1994)



\bibitem[14]{P} G. P\'olya,
\emph{\"Uber die Nullstellen gewisser ganzer Funktionen},
Math. Z. 2, pp. 352--383 (1918)


\bibitem[15]{Se} A. M. Sedletskii,
\emph{Addition to P\'olya's theorem on zeros of Fourier sine-transforms},
Integr. Transf. Spec. F. 9, pp. 65--68 (2000)

\bibitem[16]{Sb} S. Steinerberger,
\emph{Fourier uncertainty principles, scale space theory and the smoothest average},
arXiv:2005.01665v1 (2020)



\bibitem[17]{St} J. Steinig,
\emph{The real zeros of Struve's function},
SIAM J. Math. Anal. 1, pp. 365--375 (1970)


\bibitem[18]{St2} J. Steinig,
\emph{The sign of Lommel's function},
Trans. Amer. Math. Soc. 163, pp. 123--129 (1972)


\bibitem[19]{T} E. O. Tuck,
\emph{On positivity of Fourier transforms},
Bull. Austral. Math. Soc. 74, pp. 133--138 (2006)


\bibitem[20]{Wa} G. N. Watson,
\emph{A Treatise on the Theory of Bessel Functions},
Cambridge University Press, London (1995)


\bibitem[21]{Wi} R. E. Williamson,
\emph{Multiply monotone functions and their Laplace transforms},
Duke Math. J. 23, pp. 189--207 (1956)


\end{thebibliography}
\end{document}